\documentclass[reqno]{amsart}
\usepackage{hyperref}
\usepackage{epsfig}
\usepackage{graphicx}

\newcommand{\C}{\mathbb{C}}
\newcommand{\R}{\mathbb{R}}

\begin{document}
\title[\hfilneg \hfil   2N-dimensional  canonical systems and Applications ]
{ 2N-dimensional  canonical systems and Applications}

\author[K. R. Acharya, A. Ludu \hfilneg]
{Keshav Raj Acharya, Andrei Ludu}  % in alphabetical order

\address{ \newline
Department of Mathematics, Embry-Riddle Aeronautical University\\ 1 Aerospace Blvd\\ Daytona Beach, FL}
\email{kacharya@spsu.edu ; aludu@erau.edu}

\thanks{}
\subjclass[2000]{34A30, 34A12, 34B60 }
\keywords{Canonical systems, periodic canonical system, Floquet theory}

\begin{abstract}
We study the 2N-dimensional canonical systems and discuss some properties of its fundamental solution. We then discuss the Floquet theory of periodic canonical systems and observe the asymptotic behavior of its solution. Some important physical  applications of the systems are also discussed: linear stability of periodic Hamiltonian systems, position-dependent effective mass, pseudo-periodic nonlinear water waves, and Dirac systems.
\end{abstract}

\maketitle
\numberwithin{equation}{section}
\newtheorem{theorem}{Theorem}[section]
\newtheorem{lemma}[theorem]{Lemma}
\newtheorem{proposition}[theorem]{Proposition}
\newtheorem{corollary}[theorem]{Corollary}

\newtheorem{exa}[theorem]{Example}
\newtheorem{defi}[theorem]{Definition}

\allowdisplaybreaks

\section{Introduction}

A 2N- dimensional canonical system is a system of $2N$ first order differential equations of the form
\begin{equation} \label{ca}
Ju'(t)=z H(t)u(t), \quad z\in \mathbb{C}, \quad t\in [0, \infty)
\end{equation}
 where $ J =\begin{pmatrix} 0&-I\\ I&0 \end{pmatrix}, $ $I$ is an $N\times N$ identity matrix and
$H(t)$ is a $2N\times2N$ positive semidefinite matrix whose entries are locally integrable.   The complex number $z \in \C $ involved in \eqref{ca} is a spectral parameter and $u(t): [0,\infty) \rightarrow \C^{2N}$ is a vector-valued function. We also assume that there is no non-empty open interval $\mathbf I$ on which $H \equiv 0$ almost everywhere. For fixed $z$, a vector valued function $ u(t,z)= (u_1, u_2, \dots u_{2N})^t $ defined on a bounded interval $ [0,N] $  is called a solution of \eqref{ca} if $u$ satisfies \eqref{ca} and all the component functions $u_1, u_2, \dots, u_{2N}$ are absolutely continuous.

The theory of two dimensional canonical systems in which $H(t)$ is a $2\times 2$ positive semidefinite matrix and $J$ and $u$ are similarly defined, has been an important tool in the direct and inverse spectral theory of  second order differential operators. The systems is a generalization of the classical equations such as Schr\"odinger, Jacobi, Dirac, and Sturm-Liouville. The origin of  such systems goes back to 1960 when L. de Branges studied the systems in connection with Hilbert spaces of entire functions \cite{db}. Since then, there have been numerous studies about such systems extending the theory to many fields of mathematics, giving operator-theoretic point of view,  which can be found in a recent book  \cite{cre} and references therein. In addition to the applications in different fields of mathematics,  these systems have direct physical applications, some of which are mentioned in  this paper \cite{PV2}-\cite{3}.

The  extension of the theory of two dimensional canonical systems to higher dimensions has also been of general interest. For example, see  \cite{as}  for the spectral theory of  canonical differential systems. The extended theory generalizes the above mentioned classical equations in higher dimensions. In this paper, we  have considered the systems in 2N dimension. The goal is to discuss the properties of fundamental solution of \eqref{ca} when the matrix $ H(t) $ is periodic and non periodic. We also obtain the asymptotic behavior of a solution  of  \eqref{ca} when $H(t)$ is periodic.   

We organize the paper as follow: In section 2, we discuss the existence of solution of \eqref{ca} and some properties of fundamental matrix solution. We then discuss about the periodic canonical systems and obtain some important properties of solution. In section 3, we present  very interesting applications of these systems.  

\section{ Preliminaries and Results }

 A canonical system \eqref{ca} is called \textit{trace normed} if $\operatorname{tr} H(t) \equiv 1. $ A canonical system can be transformed into a trace normed  by the change of variable,  $$ x(t) = \int _0^t \operatorname{tr} H(s) ds .$$ Let $ \tilde H(x) = (\operatorname{tr} H(x))^{-1} H(x(t)) .$ Then  $ Ju' = z \tilde H(x) u $ is trace normed. Moreover, if $u(t)$ is a solution of \eqref{ca}, then $ \tilde u(x, z)= u(t(x), z)$ is a solution to $ Ju' = z \tilde H(x) u .$ Therefore, we will consider \eqref{ca} a \textit{trace normed} which will make the situation easier. In order to show the existence of a solution, we first  
  write the equations as an integral equation of the form
 \begin{align} \label{eq3} u(t) =   u(0)  -z \int_0^t J H(s)u(s) \ ds \end{align} and associate the Eq. \eqref{eq3} by an operator on a Banach space to  apply the Banach fixed point theorem. For this, let $z\in K$ a compact subset of $\C$ so that $|z|< R,\ R\neq 0 $. Let  $\mathbf I =[a, b] \subset [-\frac{1}{4R}, \ \frac{1}{4R} ]$  and  $C(\mathbf I,\  \C^{2N}) $, the set of all continuous vector-valued functions  defined on the interval $\mathbf I.$ Then the space $C(\mathbf I, \C^{2N}) $ is a Banach space with the norm $$ \| u\| = \sup_{t \in \mathbf I} |u(t)| $$ where $|u(t)|$ is the Euclidean norm in $\C^{2N}$. With this norm we have the following lemma.
 
 \begin{lemma} \label{lemma 2.1} If $  H(t) $ in \eqref{ca} is positive semidefinite matrix with $\operatorname{tr} H(t) \equiv 1,$ then $ \| H(t)u(t)\| \leq \|u\| $ for all $ u \in C(\mathbf I, \C^{2N}) $ .\end{lemma}

 \begin{proof}  Note that if a matrix is positive semidefinite matrix the Euclidean norm of the matrix is dominated by the trace of the matrix. Therefore, for any $t\in I$ we have  
 \begin{align*}  \| H(t)u(t)\| = \| H(t)\| \|u(t)\| \\ \leq \operatorname{tr} H(t) \|u\|.\end{align*}  
 Since $\operatorname{tr} H(t) \equiv 1,$ we get $ \| H(t)u(t)\| \leq \|u\| .$

\end{proof}
 
 Now we define an operator $T$ on $C(\mathbf I, \ \C^{2N}) $ by  \begin{align} (Tu) (t) =   u(a)  -z \int_a^t J H(s)u(s) \ ds .\end{align}
Clearly $T$ is a bounded linear operator. We show that $T$ is a contraction mapping on  $C(\mathbf I, \C^{2N}) .$ First we observe that 

 \begin{align*} \| Tu -Tv \| & \leq  |z|  \int_a^t \|J H(s) (u(s)-v(s)) \| \ ds .\end{align*} Since $J$ is a unitary matrix \begin{align*}\| Tu -Tv \|  \leq  |z|  \int_a^t \| H(s) (u(s)-v(s)) \| \ ds .\end{align*}
 
 Since  $H$ is positive semidefinite with $\operatorname{tr} H(t) \equiv 1$, by lemma \ref{lemma 2.1} 
 
 \begin{align*}\| Tu -Tv \|  \leq  |z|  \int_a^t \| H(s) (u(s)-v(s)) \| \ ds \leq  \int_a^t \|(u-v)\| \ ds = (t-a) \|u-v\|.\end{align*}
 
  Here $\mathbf I =[a, b] \subset [-\frac{1}{4R}, \ \frac{1}{4R} ]$ , which imply that \begin{align} \| Tu -Tv \| \leq  \frac{1}{2} \|u-v \|  .\end{align}
 Hence $T$ is a contraction mapping on the Banach space $C(\mathbf I, \C^{2N}) .$ Therefore $T$ has a fixed point say $u$. That is $T u = u $  in  $C(\mathbf I, \C^{2N}) $. So \eqref{eq3} has a solution $ u $  in  $C(\mathbf I, \C^{2N}) $. The uniqueness follows from the fact that $f(t, u)= zJH(t) u(t)$ is Lipschitz in any compact subset of  $C(\mathbf I, \C^{2N}) .$ By continuation of solutions, for example see \cite{Jack}, we have the following theorem.
 
 \begin{theorem} \label{thm1} For any $z \in \C$,  $H(t)$ is a $2N\times2N$ traced normed positive semidefinite matrix, any bounded interval $\mathbf I \subset \R $ , theere exists unique vector valued solution to the canonical system \begin{equation*} 
 	Ju'(t)=z H(t)u(t), \quad z\in \mathbb{C},\, u(t_0) =u_0
 	\end{equation*}  for all $t\in \mathbf I$ and $ u_0 \in \R^{2N}$. \end{theorem}

 \begin{theorem} For any $z \in \C, $ the set of solutions of \eqref{ca} is a vector space of dimensions $2N.$\end{theorem}
 
 \begin{proof} Let $u^i(t), i=1, \dots , 2N $  be $2N$ solutions to\eqref{ca} such that $u^i(t_0) = e_i, $ where $e_i, \, i=1, \dots, 2N$ are the basis vectors in $\C^{2N}.$ Let $a_i,\, i=1, \dots 2N $ be scalars and  \begin{equation} \label{eq7}  u(t) = a_1 u^1(t) + \dots a_{2N} u^{2N} (t)\end{equation}  Then  \begin{equation} \label{eq8}  u(t_0) = a_1 e^1 + \dots a_{2N}e^{2N} \end{equation}  So, $ u(t_0) = 0 $ if and only if $a_i =0,\, i=1, \dots 2N .$  Therefore  $u^i(t), i=1, \dots , 2N $  are linearly independent. Furthermore, since any vector $u(t_0) $ can be written as  in \eqref{eq8}, by uniqueness, each solution to \eqref{ca} can be written \eqref{eq7}. Therefore, $ \{u^i(t), i=1, \dots , 2N  \} $ is a basis of the solution space. \end{proof}

 \subsection {Matrix Solution}
 
 An $2N\times 2N$ matrix-valued function $W(t,z)$ is  called a fundamental matrix solution of \eqref{ca} if \begin{equation} \label{eq2.1}JW'(t,z) = zH(t)W(t,z) \end{equation} and for any fixed $z \in \C, $  the $2N$ columns of $W(t,z)$ are linearly independent solutions of \eqref{ca}. 
 
 \begin{theorem} Let $W(t,z)$ be a fundamental solution of \eqref{ca}. Then the determinant of $W(t,z)$ is independent of $t$. That is \begin{equation} \displaystyle \det (W(t,z)) = \det(W(t_0,z) ). \end{equation}  \end{theorem}

 \begin{proof} In order to prove this theorem we will use the Jacobi formula \begin{equation}\label{jf} \det e^A = e^{ \operatorname{tr} A }.\end{equation}  	 
 Since $W(t,z)$ is a solution of  \eqref{ca}   we have \begin{equation} \displaystyle W(t,z)^{-1} W'(t,z) = -zJH(t) .\end{equation} It follows that  \begin{equation} \displaystyle \frac{d}{dt} \ln (W(t,z)) = -zJH(t) .\end{equation}
 and on integration we get, 
\begin{equation} \displaystyle  \ln (W(t,z)) =  \ln (W(t_0,z))-z\int_{t_0}^tJH(s)ds .\end{equation}
It follows that \begin{equation*} \displaystyle   W(t,z) =  W(t_0,z)
e^{-z\int_{t_0}^tJH(s)ds }.\end{equation*}
Then by \eqref{jf} we get
 \begin{equation*} \displaystyle \det (W(t,z)) = \det(W(t_0,z) ) e^{-z \int_{t_0}^t\operatorname{tr} (JH(s) )ds}\end{equation*}
Since $H(t)$ is symmetric, $ \operatorname{tr} ( J H(t) ) \equiv 0. $ Hence \begin{equation*} \displaystyle \det (W(t,z)) = \det(W(t_0,z) ). \end{equation*}

\end{proof}

\begin{theorem} \label{fs}  The fundamental solution  $W(t,z)$ of \eqref{ca}, with the initial values $W (t_0, z) = J$  is symplectic for all $t, t_0 \in \mathcal I$ and $z\in \C.$ Conversely, if $ W (t)$ is a continuously differentiable symplectic real matrix-valued function such that $W(t)$ and  $W'(t)$ commute, and $JW'$ is positive semidefinite then $  \mathcal W (t, z )= e^{z W(t)} $ is a matrix-valued solution of a canonical system. \end{theorem}

\begin{proof} Let $W(t,z)$ be the fundamental solution of \eqref{ca} with the initial values  $W (t_0, z) = J$. Then

\begin{align*} \displaystyle \frac{d}{dt} (W(t,z)^t J W(t,z)) &=  (W(t,z)^t)' J W(t,z) + W(t,z)^t J W'(t,z) \\ & = (W'(t,z))^t JW(t,z) + W(t,z)^t zH(t) W(t,z)\\ & = (-zJH(t)W(t,z))^t JW(t,z) + z W(t,z)^t H(t) W(t,z) \\ & = -zW(t,z)H(t)J^t JW(t,z) + z W(t,z)^t H(t) W(t,z) \\ & = -zW(t,z)H(t)W(t,z) + z W(t,z)^t H(t) W(t,z)\\ & = 0 .\end{align*}
It follows that $(W(t,z)^t J W(t,z)) =  C $, a constant. Using the initial value $W (t_0, z) = J$ we bet  $C= J$ which shows $(W(t,z)^t J W(t,z)) =  J $ which is a condition for $W(t,z)$ to be symplectic matrix. Conversely, suppose  $W(t)$ is a continuously differentiable symplectic real matrix valued function. Then $\mathcal W(t)$ is a matrix solution to a canonical system $ Ju'(t) = z H(t) u(t)$ where $H(t) = JW'(t).$ 

 \end{proof}

It follows from theorem \ref{fs} that if the fundamental matrix solution $W(t,z)$ of \eqref{ca} has initial value  $W (t_0, z) = J$, then $ \displaystyle \det (W(t,z)) = 1  .$

\subsection{Periodic Canonical System and Floquet theory}

In this section we discuss $2N$ dimensional periodic canonical system:
\begin{equation} \label{pca}
Ju'(t)=z H(t)u(t), \quad z\in \mathbb{C}
\end{equation}
where $H(t)$ is a $2N\times 2N$ positive semidefinite matrix and satisfy  $H(t+p)= H(t)$ for some $p\in \R $ and all $t\in \R.$\\

\begin{lemma}If $ W (t, z)$ is a fundamental matrix solution of \eqref{pca} with initial value $W(t_0, z) = J $  then \[W(t+p, z) = -  W(t, z) J W(p,z).\] \end{lemma}

\begin{proof} Let $U(t)=W(t+p)$ and $V(t)= -W(t)JW(p).$ Then $U(t)$ and $V(t)$ are matrix solutions of \eqref{pca}  with the same initial values $U(0)= V(0) = W(p) .$ Then by theorem \ref{thm1} we have $U(t) \equiv V(t).$ Thus \[W(t+p, z) = - W(t, z) J W(p,z)\] \end{proof}

The matrix $ C= -JW(p)$ is called the \textit{monodromy matrix} of \eqref{pca}, and the eigenvalues $ \lambda$ of $C$ are called the characteristic multiplier of \eqref{pca}.

 The monodromy matrix $C$ is symplectic being the product of two symplectic matrice and if $\lambda$ is an eigen value of $C$ then $\frac{1}{ \lambda}$ is also an eigen value of $C.$ Therefore we have, 
 \[ \det{C}= \lambda_1, \lambda_2 \dots \lambda_{2N} = 1 .\]

Note that if a matrix is non singular it has an exponential form. Indeed, we can write $C$ in its Jordan normal form, $C=P J P^{-1}$, and build a new matrix $\tilde{J}$ with the same Jordan blocks structure as $J$, having $\mu_k$ instead of $\lambda_k$ in the block diagonals where $\lambda_k =e^{\mu_k}$, which is always possible since $\lambda_k \neq 0$. The matrix exp$\tilde{J}$ maintains the same Jordan blocks direct product structure, except its blocks are upper triangular, hence each of them similar to the corresponding Jordan block of $J$. It is obvious that $C$ can be written as an exponential $C=e^{P\tilde{C} P^{-1}} = e^{K}$ so that  $W(p) =  J e^{K}$ for some complex matrix $K.$ Any $\mu$ such that $\lambda = e^{\mu }$ is called the \textit{Floquet exponent.}

\begin{theorem}[The Floquet-Lyapunov theorem] The fundamental matrix solution $W(t)$ of \eqref{pca} that satisfies $W(0) =J$ is of the form \begin{equation} \label{nf} \displaystyle W(t) = J U(t)e^{\frac{1}{p}Kt}\end{equation} where $U(t)$ is symplectic and $p-$ periodic matrix-valued function and $K \in \C^{ 2N \times 2N } .$  \end{theorem}

\begin{proof} Write $ \displaystyle U(t) = -J W(t) e^{-\frac{1}{p}Kt}.$ Then $U(t)$ is symplectic because of the product of symplectic matrices. We show $U(t)$ is $p-$ periodic,  

\begin{align*} U(t+p)  & = -J W(t+p) e^{- \frac{1}{p} K(t+p)} \\ & = J W(t)J W(p)  e^{-K} e^{- \frac{1}{p}Kt} \\ & = J W(t)J J e^K e^{-K} e^{- \frac{1}{p}Kt}\\&  = -J W(t)  e^{- \frac{1}{p}Kt} \\ &= U(t) \end{align*}

It follows that $ W(t) = J U(t) e^{\frac{1}{p}Kt}$
\end{proof}

\begin{proposition} \label{ps}Let $\lambda $ be a characteristic multiplier and $ \mu$ be the corresponding characteristic exponent so that $ \lambda = e^{\mu}$, then there exists a solution $u(t,z)$ of \eqref{pca} such that 
\begin{tabbing} (a) $u(t+p) = \lambda u(t,z)$ \\

(b)  $u(t, z) = e^{\frac{\mu}{p} t}  v(t,z), $ for some periodic function $ v(t, z)$ with period $p$. \end{tabbing} \end{proposition}

\begin{proof} Let $\lambda$ is an eigen value of $C$ and $c$ is an eigenvector corresponding to eigenvalue $\lambda.$ Let $u(t,z) = W(t,z)c$ where  $ W(t,z)$ is the fundamental matrix solution of \eqref{pca}. Then \[J u'(t,z) = z H(t)u(t,z).\] 
(a) We show that $ u(t+p,z) = \lambda u(t,z) .$ 	
\begin{align*} u(t+p, z)  &= W(t+p, z) c \\
 &= W(t,z ) C c \\ &= \lambda W(t,z)c \\ &= \lambda u(t,z)\end{align*}

(b) Let $v(t,z) = u(t,z) e^{-\frac{\mu}{p} t} $. We show that $v(t,z)$ is $p \ -$ periodic.

\begin{align*} v(t+p, z) = u(t+p,z) e^{-\frac{\mu}{p} (t+p)}
= \lambda u(t,z) e^{-\frac{\mu}{p}t } e^{-\mu}  =  u(t,z) e^{-\frac{\mu}{p}t } = v(t,z)\end{align*}

\end{proof}

 Using the form \eqref{nf} of the fundamental matrix of  \eqref{pca} we have, 
 
  \begin{equation} \label{nf1} \displaystyle -U'(t)e^{\frac{1}{p} Kt} - U(t) \frac{1}{p} K e^{\frac{1}{p}Kt} = z H(t) JU(t) e^{\frac{1}{p} Kt} \end{equation}
  Multiplying from left on both sides by $e^{ -\frac{1}{p}Kt} $ we get,
  
 \begin{equation} \label{nf2} \displaystyle -U'(t) - U(t) \frac{1}{p} K = z H(t) JU(t)  \end{equation}
 
 \begin{corollary} The symplectic periodic change of variable $v(t,z) = JU(t,z)u(t,z)$ transforms the periodic canonical system \eqref{pca} to a constant linear system. \end{corollary}

 \begin{proof}
 
  The symplectic periodic change of variable $v(t,z) = JU(t,z)u(t,z)$ in the periodic canonical system \eqref{pca} yield,
  
\begin{equation} \label{nf3} \displaystyle - U'(t)u(t) - U(t) u'(t) =  z H(t) JU(t)u(t)  \end{equation}

  Applying a vector $u(t)$ from the right on \eqref{nf2} we get
   
\begin{equation} \label{nf4} \displaystyle - U'(t)u(t) - U(t) \frac{1}{p} K  u(t) =  z H(t) JU(t)u(t)  \end{equation}

Subtracting \eqref{nf3} from \eqref{nf4} we get,

\begin{equation} \label{eq5} u'(t) = \frac{1}{p} K  u(t) \end{equation}

\end{proof}

 Note  that the fundamental matrix solution of \eqref{eq5} is  $ W(t)  = e^{\frac{1}{p} Kt} .$

\begin{theorem} 
\label{theo.2.10}
For any $ z \in \C,$ there exists at least a   solutions $u(t,z)$ of \eqref{pca} that is either pseudo-periodic or  satisfy $u(t, z) \rightarrow 0 $ as $t \rightarrow \infty .$  \end{theorem}

\begin{proof}  If $\lambda $ is an eigenvalue of $C$, $ \frac{1}{\lambda}$ is also an eigenvalue of $C$. Therefore, there exists an eigenvalue $\lambda $ with the possibilities: $  | \lambda | <1  \ ,  \ | \lambda | = 1 .$ 
	
 If $| \lambda | <1 , $ then the Floquet exponent is given by $ \mu = \ln (|\lambda|) + i \arg { (\lambda ) }.$ By proposition \ref{ps}  there exists a solution of \eqref{pca} of the form   $u(t, z) = e^{\frac{\mu}{p} t}  v(t,z), $ for some periodic function $ v(t, z)$ with period $p$. Since the real part of $\mu$, $ \Re (\mu) < 0 ,$ the solution satisfy $ u(t,z) \rightarrow 0 $ as $t \rightarrow \infty .$ 

If $|\lambda| =1, $ then $ \lambda = \pm 1 $ or  $ \lambda = \pm i ,$ then $\mu$ is purely imaginary say $\lambda = e^{i \theta }$. Therefore  the solution $u(t, z)  = e^{ i \theta t}  v(t,z)$, where $v(t,z)$ is a $p-$ periodic function. So  $u(t, z)$ is pseudo-periodic solution.

\end{proof}

\section{Physical applications}

The problem of solving linear differential equations with periodic coefficients, which pretty much overlaps with the Floquet theory, is a century-old subject 
with vast applications in several areas of science and technology, ranging from quantum  to classical physics, chemistry, control theory, dynamical systems and many more \cite{PV2}-\cite{3}. It represents, for example, a powerful tool to study nonlinear perturbations, noise, and stability of systems depending instantaneously on time, and admitting periodic steady states. A very well-known application is  the study of stability of the inverted pendulum and the Hill equation. In this section we present only four selected examples of physical models that use  equations like \eqref{pca}: linear stability of periodic Hamiltonian systems, position-dependent effective mass systems, pseudo-periodic water waves, and higher dimensional Schr\"{o}dinger equation. In all these examples \eqref{pca} describes a linear Hamiltonian system with $J$, the standard symplectic matrix in $2N$ dimensions, belonging to Sp$(N,\mathbf{R})$ symplectic group and endowing the phase space with symplectic structure. 

\subsection{Linear stability of periodic Hamiltonian}
If the components of the vector-valued function $u(t) \in \mathbf{C}^{2N}$ describe the (generalized: positions and momenta) degrees of freedom  of a classical finite-dimensional Hamiltonian system, the corresponding Hamiltonian equation for the evolution of the state vector has the form
\begin{equation}
\label{eq.phys.1}
J u'(t)= \nabla_{u} \mathcal{H}(u),
\end{equation}
where $\mathcal{H}$ is the Hamiltonian smooth function of $u$, and $\nabla_{u}$ is the gradient operator. The solutions of \eqref{eq.phys.1} describe  flows with conservation of energy because the skew-symmetry property of $J$ imposes the flow of $u$ to be orthogonal to the gradient of the energy. As a consequence, any solution $u(t)$ of \eqref{eq.phys.1} corresponding to initial conditions $u(0)=u_0$ fulfills $\mathcal{H}(u(t))=\mathcal{H}(u_0)$. A very useful  problem is the study of the dynamics of solutions with initial conditions in a bounded neighborhood of the equilibrium points $\Phi(t)$ of \eqref{eq.phys.1}, that is $||u_0 - \Phi||\ll 1$, where $\nabla_{u} \mathcal{H}(\Phi)=0$. The stability of these critical points  can be analyzed by expanding the Hamiltonian in Taylor series around $\Phi + \delta \Phi$, \cite{booki}, and analyzing the linearization of the flow around $\Phi$
\begin{equation}
\label{eq.phys.2}
J (\delta \Phi)_{t}=  L \delta \Phi +\mathcal{O}(||\delta \Phi||^2),
\end{equation}  
where the linear operator $L_{ij}=-\partial^2 \mathcal{H}(\Phi) / \partial u^i \partial u^j$, $i,j=1,\dots 2N$ is the symmetric Hessian matrix for $\mathcal{H}$. This linearization is exactly our study case \eqref{ca} so we can assume $H=L$. The skew-symmetry properties of $J$ 
%($J^{-1}=J^{t}=-J$) 
and symmetry of $H$ guarantee that the point spectrum of $J^{-1}L$ is symmetric with respect to the real and imaginary axes, $ \lambda \in \hbox{Sp}(JL) \rightarrow -\lambda, \pm \lambda^{*} \in \hbox{Sp}(JL)$ and hence either the critical points are linear exponentially unstable or the linear stability problem is irrelevant (imaginary spectrum),  \cite{booki}, which means that the critical points of the linearized equation \eqref{eq.phys.2}  are not likely to be asymptotically stable. One possibility to ensure the stability of a critical point is to ask for the critical point $\Phi$ to be a non-degenerate minima of $\mathcal{H}$.  However, this would  request for $L$ to be a positive-definite matrix, case which is not covered by our system \eqref{ca} since $H$ is  requested to be positive semidefinite only. Indeed, if we consider $\alpha_{-}$ to be the lower bound of the point spectrum of $L$, it results 
$$
|| \delta \Phi (t) ||^2 \le 2\frac{H(u (t))-H(\Phi)}{\alpha_{-}}=2\frac{ H(u_0)-H(\Phi)}{\alpha_{-}},
$$
where equality is attained for the eigenvectors of $L$. If the point spectrum of $L$ contains however zeroes, the upper bound in the right hand side blows up to infinity, consequently level sets for $\mathcal{H}$ infinitesimally closed to $\Phi$ become  unbounded hyperbolas, and the critical point becomes unstable.

A stronger result can be obtained if the  linearization $L$ of the Hamiltonian $\mathcal{H}$ is a periodic function of time. In this case we can amend the linear stability problem by using  Theorem \eqref{theo.2.10}. We can write the parameter $z=\rho e^{i \phi}$ and interpret $\rho$ as a time re-scaling and the complex phase as a $U(1)$ symmetry for $H$. We have
\begin{corollary}
\label{c1}
If  \eqref{pca} describes a periodic linearization of a Hamiltonian equation \eqref{eq.phys.1} around any of its critical points, then for any complex $z$ there is at least one continuous perturbation which is either exponentially stable or marginally stable.
\end{corollary}
The proof results immediately by applying \eqref{theo.2.10}. The negative real part of the Floquet exponent for cases with $|\lambda | <1$ ensures that the perturbation approaches zero in norm faster than, or at least as fast as an exponential, so we have 
an exponentially stable critical point. For the cases with $|\lambda |=1$ the perturbation is pseudo-periodic, and being continuous it is bounded in norm, so we have a marginally stable critical point \cite{marsta}. An example from classical mechanics of such a situation is given by the
existence of pseudo-periodic solutions for the dynamics of a satellite in a nearby of a geostationary orbit. Similar pseudo-periodic solutions were obtained in \cite{strel} by using a generalized version of the Weinstein-Moser theorem. In this example, instead of periodic linear Hamiltonians, the author uses the theory of compact Lie group invariant Hamiltonian. The $z-$independence of our result in \eqref{theo.2.10}, \eqref{c1} provides in our case $U(1)$-invariance, fact  which is in agreement with Theorem 1.1 in \cite{strel} applied to compact groups. At the same time, our result is distinct from the cases studied in \cite{marsta,strel} because the periodicity of our $L$ is a non-compact symmetry, so it does not fall under the cases presented in this literature. 

\subsection{Position-dependent effective mass example}

While position-dependent (effective) mass quantum mechanical systems
have repeatedly received attention in many areas of physics due to their relevance in describing the physics of many micro-structures of current interest such as, in the last decade, in problems of quantum dots \cite{PV2}, intelligent states \cite{PV3}, Aharonov-Bohm-Coulomb systems \cite{PV4}, Fermi gas \cite{PV5}, Dirac equation \cite{PVa}, and SUSY methods for superconductor systems \cite{PV6}, etc. (see also \cite{PV7}-\cite{PV9} for theoretical developments and references therein), interest in classical problems having a position-dependent variable mass is relatively recent and rapidly developing subject \cite{PV10}-\cite{PV13}. The simplest case of a position-dependent mass 
classical oscillator has been approached by \cite{PV10,PV11}, for 
the analysis of a Duffing oscillator which provides ground for interesting effects like bifurcations, and chaos.
\begin{figure}
	\includegraphics[scale=.6]{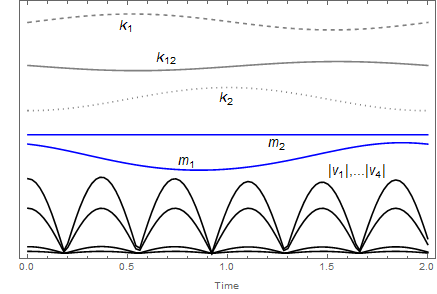}
	\centering
	\caption{First five curves from above: time dependence of the periodic ($p=2s$) entries of the $N=1$ Hamiltonian matrix \eqref{eq.hamaq}. The bottom (black) four curves represent the absolute values of the pseudo-periodic components of the solution $u(t)$, for $z=0.9, \theta=8.54$.}
	\label{fig}
\end{figure}
We introduce a classic equivalent of the \textit{de Roos} periodic time-dependent mass Hamiltonian \cite{PV7,PV11}, for two degrees of freedom $N=1$, in the form
\begin{equation}
\label{eq.hamaq}
H(t)=
\begin{pmatrix}
k_1 +k_{12} & -k_{12} & 0 & 0 \\
-k_{12} & k_2 +k_{12} & 0 & 0 \\
0 & 0 & m_{1}^{-1} & 0 \\
0 & 0 & 0 & m_{2}^{-1}
\end{pmatrix}.
\end{equation}
If this Hamiltonian is periodic, $H(t)=H(t+p)$, it can describe the evolution of two quasi-particles (impurities in super-fluid or Leidenfrost drops, or atomic clusters) under the action of a time variable (periodic) potential, with secondary effect an induced time-dependent relative interaction between the quasi-particles. In our application  we substitute the explicit position-dependence mass, with implicit time-dependence of mass and the other entries of the Hamiltonian matrix. \eqref{eq.hamaq} describes the evolution of the positions and momenta of two such quasi-particles $u(t)=(q_1, q_2, p_1, p_2)$. The  periodic coefficients $k_{1,2}(t)$ and $k_{12}(t)$ are time-variable elastic constants describing a linearized interaction between the quasi-particles and the external potential, also their mutual interaction through $q_1 - q_2$, respectively, and  $m_{1,2}(t)$ are the periodic mass coefficients. To exemplify how our results apply to this model, we build quasi-periodic solutions of the form $u(t;\theta)=e^{i \theta t} v(t)$, with periodic $v(t)=v(t+p)$, where $v$ does not dependent on $\theta$. We write \eqref{eq.hamaq} in the form $JH\rightarrow \theta  J H(t)+G_{0}(t)$ and solve the eigenproblem $JH(t) v(t)=\lambda v(t)$ asking the eigenvalue to be time- and $\theta-$independent, \textit{i.e.} $\lambda=-i z^{-1}=$const. We notice that for any periodic coefficients $k_{1,2}, m_{1,2}$ the invariance of the spectrum is guaranteed by the constraint
$$
k_{12}=\frac{z^2 k_1 k_2 -k_2 m_1 - k_1 m_2 +m_1 m_2 z^{-2}}{m_1 +m_2 -z^2 (k_1 + k_2)}.
$$
By implementing this value for $k_{12}$ in the eigenvectors $v$ we can find  solutions for $G_0$ from the condition $v'=-z G_0 v$. For a Hamiltonian given by $J H(t)+G_{0}(t)$ with $G_0$ obtained as above, the corresponding $u(t;\theta)$ solution of the system \eqref{pca} is indeed pseudo-periodic. The exact expression of $u$ is too long to be presented here, but we show in Fig. \eqref{fig} the time-dependence of the Hamiltonian matrix entries of period $p=2s$, and of the components of the solution $u(t)$ whose period spectrum has $p\pm \theta$ resonances. We mention that our construction is $p-$ and $\theta-$independent, through Theorem \eqref{pca}, so it represents a general approach, demonstrating the existence of the pseudo-periodic solutions for this type of Hamiltonians.

\subsection{Nonlinear water waves}

Nonlinear traveling waves on the free surface of an ideal fluid are the focus of a huge number of papers on the subject. The strongest verification of any analytic or numeric solution in terms of such waves is to show its matching with numerical solutions of the Euler equations for ideal fluids. By using a reformulation of the Euler equations for capillary gravity waves, involving surface integrals and a variational argument inspired by the Weinstein-Moser theorem, the authors in \cite{craig} were able to provide rigorous existence criteria for periodic traveling wave solutions in two and three dimensions. In two dimensions this was proven by a  straightforward approach of the theorems of Levi-Civita-Struik \cite{waw}. For three dimensions,
on top of the existence of periodic traveling capillary gravity water waves obtained in \cite{craig}, the existence of pseudo-periodic waves was also proved in \cite{quasi}. These results were obtained  by using of the resonant Lyapunov center theorem coupled with the Lyapunov–Schmidt method. 

However, the linearized Hamiltonians in both cases are constant. The periodicity of the solutions in \cite{craig} occurs from the truncation of the space in periodic intervals, using periodic boundary conditions and hence working the system, reduced now to the dynamics of the free surface,  on a torus manifold. The quasi-periodic waves from \cite{quasi} were obtained through the linearization of the Euler equation Hamiltonian by using a pseudo-periodic operator built in the KAM theory sense. Both these directions however, avoided the situation in which
the linearized operators have zero eigenvalue of higher multiplicity, thus involving the occurrence of small denominator problems. These are situations where more than  one or two solutions of the linearized equation have the same phase velocity, situations which actually are common in real experiments, especially when the surface tension of the fluid becomes very large (the Bond number approaches zero).

In our example we use a different approach towards the linearization of the Euler equations, allowing the phase velocity of the waves to be time dependent, like for example in long waves in the Benjamin-Fair Modulation Instability situations \cite{BFI}. By using the results proved in section 2 we demonstrate the existence of pseudo-periodic water waves, in agreement with cited literature. For a two dimensional body of water with finite depth $h$ Euler equations reduce the Bernoulli equation for the water velocity potential $\phi(x,y,t)$
\begin{equation}
\label{eq.model}
\phi_t +\frac{1}{2}|\nabla \phi|^2 +g \eta -\sigma \nabla \cdot \biggl[ \frac{\eta_x}{\sqrt{1+\eta_{x}^2}}\biggr]=0, \hbox{at} \ y=\eta(x,t),
\end{equation}
where the free surface of the water is described by $\eta(x,t)$, $\sigma$ is the surface tension coefficient and $g$ is gravitational acceleration. The water velocity  is potential $\vec{V}=\nabla \phi$, and we add to \eqref{eq.model} boundary conditions at the rigid bottom $\phi_y =0$ at $y=-h$, and nonlinear kinematic free surface boundary condition
\begin{equation}
\label{eq.bc}
\eta_t +\phi_x \eta_x -\phi_y =0, \hbox{at} \ y=\eta(x,t).
\end{equation}
When the free surface $\eta(x,t)$ and Dirichlet boundary condition at water surface for $\tilde{\phi}(x,t)=\phi(x,\eta(x,t),t)$ are given, one can  solve the full problem, since $\phi$ satisfies Laplace’s equation. In this way the water wave problem reduces  to a Hamiltonian system in the surface variables $\eta$ and   $\tilde{\phi}$ 
which are canonically conjugate \cite{craig}. The linearized Hamiltonian for the Fourier transform $F[\cdot]$ of the canonical variables becomes
\begin{equation}
\label{eq.hama}
\begin{pmatrix}
F[\eta]_t  \\
F[\tilde{\phi}]_t 
\end{pmatrix}
=\begin{pmatrix}
g+\sigma k^2 & -i \vec{c}(t)\cdot \vec{k} \\
i \vec{c}(t)\cdot \vec{k} & k \tanh{h k}
\end{pmatrix}
=
\begin{pmatrix}
F[\eta]  \\
F[\tilde{\phi}] 
\end{pmatrix},
\end{equation} 
where $\vec{c}=(c_x,c_y )$ is the wave phase velocity which is considered here time dependent as in the most general situation of mixing of waves \cite{BFI}, and $\vec{k}=(k_x,k_y )$ is the wave vector. The system \eqref{eq.hama} is a $2 \times 2$ linear Hamiltonian system like the one in \eqref{pca} if we assume periodic function for the phase velocity. In this case \eqref{eq.hama} obeys Theorem \eqref{theo.2.10}, namely, for appropriate initial conditions  it  accepts pseudo-periodic solutions in velocity potential and surface function. The role for the arbitrary parameter $z$ can be attributed to the surface tension coefficient $\sigma$, or the wavelength $k$. This direct application is in agreement with the KAM theory for capillary-gravity water waves presented in \cite{craig,quasi}. Our example  is rather general, and not  limited to water waves only, for it can be applied to other dynamical systems like nonlinear waves on networks \cite{waterx}.

\subsection{Dirac systems in semiconductors}

The effect of the dimensions of space upon physics laws is a modern question  having its roots in the efforts of unification of gravity with the standard model, and its potential future in understanding of the dark matter and dark energy puzzles. Extra space dimensions ($N>3$) must be somehow of negligible "size", that is compactified, otherwise we would have seen them already. Nevertheless, higher dimensions for Hamiltonian systems introduce consequences for the stability of matter, question yet without an answer. A similar higher dimensional approach uses the $N-$dimensional Schr\"{o}dinger equation  \cite{41}, in the study of the hydrogen atom in five dimensions, the isotropic oscillator in eight dimensions, or the position and momentum information entropies of $N-$dimensional systems. Recently,  motivational problems towards investigating $n-$dimensional Schr\"{o}dinger equations are inspired by the dimensional expansion technique used to obtain nonperturbative results in quantum field theory where the space-time dimension is used as an expansion parameter, most specifically in the case of renormalization  for a self-interacting scalar quantum field theory in the Ising limit \cite{5}. The oscillatory properties of the solutions of equation (1.1) are involved in a number of problems in  quantum mechanics in the phase space (Wigner formulation)  \cite{4}, and  more recently in the fields of graphene\cite{D1,D4}, electronic materials and Dirac systems \cite{PVa, D2}, and black holes physics \cite{D3}.

Since the advent of graphene as a $2-$dimensional electronic material which can be produced in the laboratory, the possibility of exploiting the valley degree of freedom within it, \cite{2,D1,D4}, and other Dirac systems has been vigorously studied \cite{4}. In order to optically-control electronic semiconductor structures one needs high intensity light exposure of the sample. This light determines the electronic struture to be represented by eigenvalues of a Floquet Hamiltonian. This new nano-science field, vallytronics \cite{1}, is a procedure of quantum manipulation of energy valleys in semiconductors, including quantum computation with valley-based qubits or other forms of quantum electronics.  In analogy to spintronics where the internal degree of freedom of spin is harnessed to store, manipulate and read out bits of information, in valleytronics similar tasks are performed using the multiple extrema of the band structure, so that the information of $0$s and $1$s would be stored as different discrete values of the crystal momentum. For appropriate choices of light beam amplitudes and phase offset one can produce a Floquet energy spectrum which has a gap for one valley but no gap for the other. If an electric current passes through this probe from electrodes  with chemical potentials tunned to have values in this gap of quasi-energy, then the system becomes Floquet-engineered.

In the following example we focus on  Two-dimensional Dirac system. In some time-periodic potential  the electronic states of the semiconductor are described by the time-dependent Schr\"{o}dinger equation. If we denote by $\alpha$ the multi-label of all electronic quantum numbers the solutions obtained from the Floquet theorem have the form
$$
\psi_{\alpha}(\vec{r},t)=u_{\alpha}(\vec{k},t) e^{i (\vec{k} \cdot \vec{r} -\epsilon_{\alpha}t)},
$$
where $T$ is the period of the Hamiltonian $H(t)$, the functions $u_{\alpha}(t+T)=u_{\alpha}(t)$ are the periodic part of the Floquet solution and at the same time the eigenfunctions of the  Floque Hamiltonian $H_{F}(t)=H(t)-i\partial_{t}$ while $\epsilon_{\alpha}$ are the quasi-energies levels and the corresponding eigenvalues of $H_{F}(t)$. The parameters are actually the Floquet characteristic exponents (unique modulo multiples of $\Omega$).
These wavefunctions are two-component spinors encoding the wavefunction
amplitudes on each of the sublattices of the honeycomb lattice. The field is assumed uniform in the plane so the electronic states are characterized by a two-dimensional wave vector $\vec{k}$ within the hexagonal Brillouin zone. The Floquet Hamiltonian has the form (in $\hbar=1$ convention) \cite{1}-\cite{3}
$$
H_{F}(\vec{k},t)=
\begin{pmatrix}
-i\partial_{t} & -\gamma Z(\vec{k},t)  \\
-\gamma Z^{*}(\vec{k},t) & -i\partial_{t} 
\end{pmatrix},
$$
where $\vec{k}$ is the wave vector of the state in the Brillouin zone, and $\gamma$ is the hopping amplitude for electrons in the graphene honeycomb lattice.
The operator $Z$ is given by
$$
Z(\vec{k},t)=\sum_{n=1}^{3}\hbox{exp }\biggl[ i \biggl( \vec{k}+\frac{e}{c} \vec{A}(t) \biggr) \cdot \vec{a}_{n} \biggr],
$$
where $\vec{a}_{n}$ are the nearest neighbors vector of a site on the lattice, and $\vec{A}(t)$ is the oscillating vector potential for the electric field in the light beam and it represents the periodic part of the Hamiltonian $\vec{A}(t)=A_{0} (\cos \Omega t, \sin \Omega t)$ with $\Omega=2 \pi/T$. The Floquet Hamiltonian fulfills the equation $H_{F}(\vec{k},t) u_{\alpha}(\vec{k},t)=\epsilon_{\alpha}(\vec{k}) u_{\alpha}(\vec{k},t)$. Moreover, we notice that the Floquet modes obey the matching relation
$$
u_{\alpha'}(\vec{k},t)=u_{\alpha}(\vec{k},t) e^{i n \Omega t}=u_{n \alpha}(\vec{k},t), \ \ n \in \mathcal{Z},
$$
while the shifted quasienergy values obey 
$$
\epsilon_{\alpha}\rightarrow \epsilon_{\alpha'}=\epsilon_{\alpha}+n \Omega=\epsilon_{n\alpha}.
$$
Hence the label $\alpha$ corresponds to a whole class of solutions indexed $\alpha'=(\alpha, n), \ n \in \mathcal{Z}$ and thus the eigenvalues can be mapped into the first Brillouin zone obeying $-\Omega/2 \le \epsilon < \Omega/2$.

For some special values of the wave vector $\vec{k}$, two symmetrically points placed along $y-$axis in the Brillouin zone by the nearest neighbors, the lattice symmetry combines with the time periodicity and hence the  relevant frequency for the Floquet problem is modified by a constant phase factor in front of $Z$. In such situations distinct states can cross at the Floquet zone boundary $\epsilon_{\alpha}=\pm \Omega/2$ (shown above) leading to topological transitions in the quasienergy band structure. This physical situation is an example of occurrence of the pseudo-periodic solutions for this type of Hamiltonian described  in Theorem \eqref{theo.2.10}.

We mention that the full vector space of the electron states is needed to allow matching of wavefunctions to external states at the boundary of the semiconductor sample. Thus, to deal with such systems one must treat the time degree of freedom as a
genuine extra dimension, compactified via the periodic
(temporal) boundary condition, which is very much in the spirit of our comment in the beginning of this section about the compactified (negligible "size") extra space dimensions. In such composite Hilbert space $\mathbf{R}^2 \otimes \mathcal{T}$ of square integrable functions on configuration space and the space $\mathcal{T}$ of functions periodic of period $2 \pi/\Omega$ the quasienergy does not depend on an extra constant phase in the electric field perturbation (the phase is equivalent to a shift in time origin), while the time-dependent Floquet functions depend on such phase at any given fixed time \cite{2,3}. Consequently, the quasienergy eigenvalue equation has the form of the time independent Schr\"{o}dinger equation in the composite Hilbert space. This feature reveals the advantage of the Floquet formalism.

\end{document}